\documentclass{amsart}
\usepackage[utf8]{inputenc}
\usepackage{amsmath}
\usepackage{amssymb}
\usepackage{caption}
\usepackage{amsthm}
\usepackage[hidelinks]{hyperref}
\usepackage{cleveref}
\crefname{lemma}{Lemma}{Lemmas}
\crefname{theorem}{Theorem}{Theorems}
\usepackage{xcolor}
\usepackage{comment}
\usepackage{graphicx}
\usepackage{tikz}
\makeatletter
\def\@settitle{\begin{center}%
  \baselineskip14\p@\relax
  \bfseries
  \uppercasenonmath\@title
  \@title
  \ifx\@subtitle\@empty\else
     \\[1ex]\uppercasenonmath\@subtitle
     \footnotesize\mdseries\@subtitle
  \fi
  \end{center}%
}
\def\subtitle#1{\gdef\@subtitle{#1}}
\def\@subtitle{}
\makeatother

\newtheorem{theorem}{Theorem}[section]
\newtheorem{corollary}[theorem]{Corollary}
\newtheorem{proposition}[theorem]{Proposition}
\newtheorem{lemma}[theorem]{Lemma}

\theoremstyle{remark}
\newtheorem{remark}[theorem]{Remark}

\newcommand{\vol}{\rm{vol}}

\DeclareMathOperator{\csch}{csch}

\usepackage{chngcntr}
\usepackage{graphicx} 
\usepackage{float}
\counterwithout{equation}{section}
\counterwithout{theorem}{section}

\begin{document}

\title{On reflections of congruence hyperbolic manifolds}

\author{Sami Douba}
\address{Mathematisches Institut der Universit\"at Bonn, Endenicher Allee 60, 53115 Bonn, Germany}
\email{douba@math.uni-bonn.de}

\author{Franco Vargas Pallete}
\address{IMPA - Instituto Nacional de Matem\'atica Pura e Aplicada, Rio de Janeiro, RJ, Brasil, 22460-320}
\email{vargas.pallete@impa.br}

\begin{abstract}
We show that the standard method for constructing closed hyperbolic manifolds of arbitrary dimension possessing reflective symmetries typically produces reflections whose fixed point sets are nonseparating.
\end{abstract}

\maketitle

Given a complete oriented Riemannian manifold $M$, we will call an isometric involution $\tau: M \rightarrow M$ a {\em reflection} if the fixed point set $\mathrm{Fix}(\tau)$ of $\tau$ is a (possibly disconnected, but nonempty) $2$-sided  hypersurface of $M$. Note that $\mathrm{Fix}(\tau)$ is then necessarily  totally geodesic. For a closed nonpositively curved manifold~$M$, denote by $\mathrm{sys}(M)$ the  systole of $M$, that is, the length of a shortest closed geodesic in $M$, and by $$h(M) : = \inf_\Omega \frac{\mathrm{area}(\partial \Omega)}{\vol(\Omega)}$$ the Cheeger constant of $M$, where the infimum is taken over all domains $\Omega \subset M$ with rectifiable boundary and of volume at most $\frac{1}{2}\mathrm{vol}(M)$. One says a sequence $(M_i)_{i \in \mathbb{N}}$ of such manifolds $M_i$ with $\mathrm{vol}(M_i) \rightarrow \infty$ is {\em expander} if there is some $\epsilon > 0$ such that $h(M_i)> \epsilon$ for all $i \in \mathbb{N}$. The purpose of this note is to observe the following.

\begin{theorem}\label{thm:main}
Suppose $(M_i)_{i \in \mathbb{N}}$ is an expander sequence of closed hyperbolic manifolds $M_i$ each admitting a reflection $\tau_i : M_i \rightarrow M_i$, and with $\mathrm{sys}(M_i) \rightarrow \infty$. Then the (possibly disconnected) hypersurface $\mathrm{Fix}(\tau_i)$ of $M_i$ does not separate $M_i$ for all sufficiently large $i \in \mathbb{N}$.
\end{theorem}

While it may seem difficult to arrange for all the conditions of Theorem \ref{thm:main} to hold simultaneously, the standard arithmetic method for constructing closed hyperbolic manifolds admitting reflections indeed results in sequences of manifolds as in Theorem \ref{thm:main}; see Remark \ref{rmk:arithmetic}. We felt compelled to share Theorem~\ref{thm:main}, despite its rather straightforward proof, as there appear to be misconceptions in the literature (for instance, in certain subsequent accounts\footnote{See, for instance, \cite[\S3]{MR2350468}, where it appears to be assumed that each reflection in the dihedral group of symmetries of the input arithmetic hyperbolic manifold separates. This subtlety is resolved in a footnote in \cite{MR4906330}. That this subtlety indeed had to be addressed is indicated by our~Theorem~\ref{thm:main}.} of the Gromov--Thurston construction~\cite{MR892185} of exotic negatively curved manifolds) regarding the content of Remark~\ref{rmk:arithmetic}. This confusion is perhaps due to the misleading cartoons often used to depict reflections. 

We also remark that it follows from superstrong approximation \cite{MR3000503, MR3445486} that, given any closed hyperbolic orbifold $M = \Gamma \backslash \mathbb{H}^n$, arithmetic or otherwise, there is an expander sequence of regular oriented finite manifold covers $M_i \rightarrow M$ with $\mathrm{sys}(M_i) \rightarrow \infty$. Thus, for instance, if~$\Gamma$ contains a reflection of $\mathbb{H}^n$, then the $M_i$, and the reflections $\tau_i$ induced on the $M_i$, are as in Theorem~\ref{thm:main}.

The proof of Theorem~\ref{thm:main} uses the following (known) elementary estimate.

\begin{lemma}\label{lem:collar}
    Let $M$ be a closed hyperbolic manifold and $\Sigma\subset M$ be a (possibly disconnected) separating closed embedded totally geodesic hypersurface. If $\Sigma$ has an open embedded collar of width $w>0$ in $M$, then $h(M)\leq\csch(w)$.
\end{lemma}
\begin{proof}
    The open collar $C_w:= \Sigma\times(-w,w)\subset M$ of width $w$ around $\Sigma$ has metric
    \[
    \cosh^2(t)g_\Sigma(x) + dt^2,\quad x\in \Sigma, -w<t<w,
    \]
    where $g_\Sigma$ denotes the metric on $\Sigma$ and $t$ denotes the distance to $\Sigma$.
    By the coarea formula, each component of $C_w\setminus \Sigma$ has volume
    
    \begin{equation}\label{eq:collarvolume}
    \int_0^w \cosh(t)\mathrm{area}(\Sigma) dt = \sinh(w)\mathrm{area}(\Sigma).
    \end{equation}
    Hence, for each component $\Omega$ of $M\setminus\Sigma$ we have that $\frac{\mathrm{area}(\partial\Omega)}{\vol(\Omega)}\leq \csch(w)$, and the lemma follows.
\end{proof}

\begin{proof}[Proof~of~Theorem~\ref{thm:main}]
 A shortest orthogeodesic segment $\omega_i$ to $\Sigma_i:=\mathrm{Fix}(\tau_i)$ in $M_i$ has length at least $\frac{\mathrm{sys}(M_i)}{2}$, as $\omega_i \cup \tau_i(\omega_i)$ is a closed geodesic in $M_i$. It follows that $\Sigma_i$ has an open embedded collar of width $\frac{\mathrm{sys}(M_i)}{4}$ in $M_i$. Now suppose (up to extraction) that $\Sigma_i$ separates $M_i$ for all $i \in \mathbb{N}$. We then obtain from Lemma \ref{lem:collar} that $h(M_i) \rightarrow 0$, a contradiction.
\end{proof}

\begin{remark}\label{rmk:quantifiers}
    Note that the above proof in fact shows that if $h(M_i)>\epsilon$, then the conclusion of Theorem~\ref{thm:main} holds as soon as $\mathrm{sys}(M_i)\geq 4\csch^{-1}(\epsilon)$.
\end{remark}

\begin{remark}\label{rmk:negativecurvature}
    Observe that, for its application in the proof of Theorem~\ref{thm:main}, we need only that Equation \eqref{eq:collarvolume} in the proof of Lemma \ref{lem:collar}  provide a lower bound for the volume of a collar in terms of the area of $\Sigma$ and a multiplicative factor that depends only on the width $w$. The latter holds more generally in the nonpositively curved setting by the Rauch--Berger comparison theorems, so that if $M$ is a closed nonpositively curved Riemannian manifold containing a separating closed embedded totally geodesic hypersurface $\Sigma$ with an open embedded collar of width $w>0$, then ${h(M)\leq w}$. Similarly, Theorem \ref{thm:main} persists if ``hyperbolic'' is replaced with ``nonpositively curved.'' In analogy with Remark \ref{rmk:quantifiers}, we have in the latter case that if $h(M_i)>\epsilon$, then the conclusion of Theorem \ref{thm:main} holds as soon as $\mathrm{sys}(M_i)\geq 4\epsilon^{-1}$.
\end{remark}

\begin{remark}\label{rmk:arithmetic}
Let $n \geq 2$, let $k \subset \mathbb{R}$ be a totally real number field with ring of integers~$\mathcal{O}_k$, and let~$f$ be a quadratic form of signature $(n,1)$ on $\mathbb{R}^{n+1}$ with coefficients in~$k$ such that~$f^\sigma$ is positive definite for each nonidentity embedding $\sigma: k \rightarrow \mathbb{R}$. Suppose moreover that $f$ is anisotropic over $k$ (note that this is automatic if $k \neq \mathbb{Q}$). We identify the level set $\{x\in \mathbb{R}^{n+1}\vert f(x)=-1, x_{n+1}>0\}$ with $n$-dimensional hyperbolic space~$\mathbb{H}^n$ and $\mathrm{O}'_f(\mathbb{R})$ with $\mathrm{Isom}(\mathbb{H}^n)$, where $\mathrm{O}'_f(\mathbb{R})$ denotes the index-$2$ subgroup of $\mathrm{O}_f(\mathbb{R})$ preserving $\mathbb{H}^n$. By the Borel--Harish-Chandra theorem \cite{MR147566}, we have that $\Gamma:= \mathrm{O}'_f(\mathcal{O}_k)$ is a uniform lattice in $\mathrm{Isom}(\mathbb{H}^n)$. We suppose henceforth that $\Gamma$ contains a reflection $\tau$ of $\mathbb{H}^n$ (this is true, for instance, if the form $f$ is diagonal).

Now let $I_1 \supset I_2 \supset \ldots$ be a sequence of nonzero ideals in $\mathcal{O}_k$ with trivial intersection, and for $i \geq 1$, denote by $\Gamma(I_i)$ the principal congruence subgroup of~$\Gamma$ of level~$I_i$. There is some $N > 0$ such that for all $i \geq N$, the quotient $M_i:= \Gamma(I_i) \backslash \mathbb{H}^n$ is an oriented manifold, and the fixed point set $\Sigma_i := \mathrm{Fix}(\tau_i)$ of the reflection $\tau_i$ of $M_i$ induced by $\tau$ is a (possibly disconnected) $2$-sided closed embedded totally geodesic hypersurface of $M_i$. Since $\bigcap_{i} I_i = \{0\}$, we have $\mathrm{sys}(M_i) \rightarrow \infty$. Moreover, it is well known that there is some $\epsilon >0$ such that $h(M_i) > \epsilon$ for $i \geq N$; see Burger and Sarnak~\cite{MR1123369} and the references therein. It then follows from Remark~\ref{rmk:quantifiers} that if~$N_0 \geq N$ is chosen such that $\mathrm{sys}(M_{N_0}) \geq 4\csch^{-1}(\epsilon)$, then $\Sigma_i$ does not separate~$M_i$ for all~$i \geq N_0$. 
\end{remark}

In the notation of Remark~\ref{rmk:arithmetic}, for $i \geq N_0$, let $p_i$ be the projection $M_i \rightarrow M_{N_0}$. Given an embedded collar $C$ around $\Sigma_{N_0}$ in $M_{N_0}$, note that $p_i^{-1}(C)$ is an embedded collar of the same width around $p_i^{-1}(\Sigma_{N_0})$ in $M_i$. As in the proof of Theorem~\ref{thm:main}, it then follows again from Lemma~\ref{lem:collar} that in fact $p_i^{-1}(\Sigma_{N_0})$ does not separate $M_i$ for $i \geq N_0$. Note that the latter is stronger than saying that $\Sigma_i$ does not separate~$M_i$ for $i \geq N_0$, since $\Sigma_i \subset p_i^{-1}(\Sigma_{N_0})$. On the other hand, it is known (see the final page in the proof of \cite[Thm.~4.1]{MR4768580}) that the number of components of $p_i^{-1}(\Sigma_{N_0})$ approaches infinity as $i \rightarrow \infty$. Hence, given $r \in \mathbb{N}$, for $i$ sufficiently large, one can express $\pi_1(M_i)$ as the fundamental group of a graph of groups whose underlying graph is a rose with $\geq r$ petals. One thus recovers the following result of Lubotzky~\cite{MR1390750}.

\begin{corollary}\label{cor:free}
For any $r \in \mathbb{N}$, there is some $i \geq N_0$ such that $\Gamma(I_i)$ admits an epimorphism onto a free group of rank $r$.
\end{corollary}

While there are by now several known arguments for the existence of a finite-index subgroup of $\Gamma$ admitting an epimorphism onto a nonabelian free group, Corollary~\ref{cor:free} guarantees the existence of a {\em congruence} such subgroup, as does Lubotzky's original argument. We highlight also that Remark \ref{rmk:arithmetic} and Corollary \ref{cor:free} concern the nontriviality in homology of {\em possibly disconnected} totally geodesic hypersurfaces of congruence hyperbolic manifolds. This distinguishes the discussion above from classical work of Millson \cite{MR422501} (see also Millson--Raghunathan \cite{MR592256}), where this question was considered for connected hypersurfaces.

We conclude with the following proposition, which shows that reflections with nonseparating fixed point sets occur quite generally. 

\begin{proposition}\label{prop:negativelycurved}
    Let $M$ be a closed negatively curved manifold admitting a reflection~$\tau$, and suppose that $ \pi_1(M)$ is residually finite. Then $M$ admits a finite cover to which $\tau$ lifts as a reflection whose fixed point set is nonseparating.
\end{proposition}
\begin{proof}
We may assume that $\mathrm{Fix}(\tau)$ separates $M$. Let $M_\mathrm{cut}$ be one of the two isometric components of the manifold with totally geodesic boundary obtained by cutting~$M$ along $\mathrm{Fix}(\tau)$, and let $\Sigma \subset M_\mathrm{cut}$ be a component of $\partial M_\mathrm{cut}$.

We first observe that, since~$\pi_1(\Sigma)$ is the fixed point set of an automorphism of~$\pi_1(M)$ (corresponding to a lift of~$\tau$ to $\widetilde{M}$), we have that ~$\pi_1(\Sigma)$ is separable in~$\pi_1(M)$, and hence in $\pi_1(M_\mathrm{cut})$.  Indeed, the fixed point set of a continuous self-map of a Hausdorff space is closed, and since an automorphism of a group $\Gamma$ is continuous with respect to the profinite topology on $\Gamma$, we have that the fixed point set of an automorphism of $\Gamma$ is closed for the profinite topology as soon as the latter topology is Hausdorff. On the other hand, to say that a group $\Gamma$ is residually finite (respectively, to say that a subgroup $\Delta < \Gamma$ is separable in $\Gamma$) is precisely to say that the profinite topology on $\Gamma$ is Hausdorff (resp., to say that $\Delta$ is closed in $\Gamma$ for the Hausdorff topology on the latter).  

We next observe that, since $M$ is closed and negatively curved, the index of~$\pi_1(\Sigma)$ in $\pi_1(M_\mathrm{cut})$ is infinite. Indeed, let $\widetilde{\Sigma}$ be a lift of $\Sigma$ to $\widetilde{M}$, so that $\widetilde{\Sigma}$ is a boundary component of some lift $\widetilde{M_\mathrm{cut}}$ of $M_\mathrm{cut}$ to $\widetilde{M}$. Then there is some boundary component~$\Sigma'$ of $M_\mathrm{cut}$ and some lift $\widetilde{\Sigma'}$ of $\Sigma'$ to $\widetilde{M}$ such that $\widetilde{\Sigma'}$ is distinct (and hence disjoint) from $\widetilde{\Sigma}$; this is immediate if $\partial M_\mathrm{cut}$ is not connected, and otherwise follows for instance from the fact that $\pi_1(\Sigma)$ is of strictly smaller cohomological dimension than, and hence of infinite index in, $\pi_1(M)$. By compactness of $\Sigma$, there is moreover such a lift $\widetilde{\Sigma'}$ whose distance from $\widetilde{\Sigma}$ is minimal. Up to replacing a closest lift $\widetilde{\Sigma'}$ with its reflection through $\widetilde{\Sigma}$, one then has $\widetilde{\Sigma'} \subset \widetilde{M_\mathrm{cut}}$. Now since~$\widetilde{M}$ is negatively curved, there is a unique point $o \in \widetilde{\Sigma'}$ that is closest to $\widetilde{\Sigma}$, and since the orbit of~$o$ under $\pi_1(\Sigma')$ is infinite, we conclude that $\widetilde{\Sigma}$  has infinitely many translates under~$\pi_1(\Sigma')$, so that there are indeed infinitely many lifts of $\Sigma$ to $\widetilde{M_\mathrm{cut}}$. 

As $\pi_1(\Sigma)$ is a separable infinite-index subgroup of $\pi_1(M_\mathrm{cut})$, we may find a finite cover $\hat{M}_\mathrm{cut}$ of $M_\mathrm{cut}$ to which~$\Sigma$ has at least $3$ lifts, so that $\partial \hat{M}_\mathrm{cut}$ has $r \geq 3$ components. Let $\hat{M}$ be the finite cover of~$M$ obtained by doubling $\hat{M}_\mathrm{cut}$ across $\partial \hat{M}_\mathrm{cut}$. Then $\tau$ lifts to a reflection~$\hat{\tau}$ of $\hat{M}$ whose fixed point set is $\partial \hat{M}_\mathrm{cut}$.

\begin{figure}
\centering

\tikzset{every picture/.style={line width=0.75pt}} 

\begin{tikzpicture}[x=0.75pt,y=0.75pt,yscale=-1,xscale=1]

\draw [color={rgb, 255:red, 208; green, 2; blue, 27 }  ,draw opacity=1 ]   (113,151) -- (265.33,151.8) ;
\draw [color={rgb, 255:red, 74; green, 144; blue, 226 }  ,draw opacity=1 ]   (113,151) .. controls (152.33,119.8) and (232.33,120.8) .. (265.33,151.8) ;
\draw [color={rgb, 255:red, 126; green, 211; blue, 33 }  ,draw opacity=1 ]   (113,151) .. controls (152.33,180.8) and (227.33,182.8) .. (265.33,151.8) ;
\draw    (113,151) .. controls (136.33,215.8) and (246.33,210.8) .. (265.33,151.8) ;
\draw [shift={(265.33,151.8)}, rotate = 287.85] [color={rgb, 255:red, 0; green, 0; blue, 0 }  ][fill={rgb, 255:red, 0; green, 0; blue, 0 }  ][line width=0.75]      (0, 0) circle [x radius= 3.35, y radius= 3.35]   ;
\draw [shift={(113,151)}, rotate = 70.2] [color={rgb, 255:red, 0; green, 0; blue, 0 }  ][fill={rgb, 255:red, 0; green, 0; blue, 0 }  ][line width=0.75]      (0, 0) circle [x radius= 3.35, y radius= 3.35]   ;
\draw  [fill={rgb, 255:red, 255; green, 255; blue, 255 }  ,fill opacity=1 ] (164,187) -- (211.33,187) -- (211.33,207.8) -- (164,207.8) -- cycle ;
\draw [color={rgb, 255:red, 74; green, 144; blue, 226 }  ,draw opacity=1 ]   (331,70) -- (480.33,70.8) ;
\draw    (331,70) .. controls (373.33,88.8) and (456.33,85.8) .. (480.33,70.8) ;
\draw [shift={(480.33,70.8)}, rotate = 327.99] [color={rgb, 255:red, 0; green, 0; blue, 0 }  ][fill={rgb, 255:red, 0; green, 0; blue, 0 }  ][line width=0.75]      (0, 0) circle [x radius= 3.35, y radius= 3.35]   ;
\draw [shift={(331,70)}, rotate = 23.95] [color={rgb, 255:red, 0; green, 0; blue, 0 }  ][fill={rgb, 255:red, 0; green, 0; blue, 0 }  ][line width=0.75]      (0, 0) circle [x radius= 3.35, y radius= 3.35]   ;
\draw  [fill={rgb, 255:red, 255; green, 255; blue, 255 }  ,fill opacity=1 ] (381,78) -- (426.33,78) -- (426.33,95.8) -- (381,95.8) -- cycle ;
\draw [color={rgb, 255:red, 126; green, 211; blue, 33 }  ,draw opacity=1 ]   (331,70) -- (479.33,162.8) ;
\draw [color={rgb, 255:red, 126; green, 211; blue, 33 }  ,draw opacity=1 ]   (480.33,70.8) -- (330,162) ;
\draw [color={rgb, 255:red, 208; green, 2; blue, 27 }  ,draw opacity=1 ]   (331,162) -- (480.33,162.8) ;
\draw    (331,162) .. controls (373.33,180.8) and (456.33,177.8) .. (480.33,162.8) ;
\draw [shift={(480.33,162.8)}, rotate = 327.99] [color={rgb, 255:red, 0; green, 0; blue, 0 }  ][fill={rgb, 255:red, 0; green, 0; blue, 0 }  ][line width=0.75]      (0, 0) circle [x radius= 3.35, y radius= 3.35]   ;
\draw [shift={(331,162)}, rotate = 23.95] [color={rgb, 255:red, 0; green, 0; blue, 0 }  ][fill={rgb, 255:red, 0; green, 0; blue, 0 }  ][line width=0.75]      (0, 0) circle [x radius= 3.35, y radius= 3.35]   ;
\draw [color={rgb, 255:red, 126; green, 211; blue, 33 }  ,draw opacity=1 ]   (332,249) -- (481.33,249.8) ;
\draw    (332,249) .. controls (374.33,267.8) and (457.33,264.8) .. (481.33,249.8) ;
\draw [shift={(481.33,249.8)}, rotate = 327.99] [color={rgb, 255:red, 0; green, 0; blue, 0 }  ][fill={rgb, 255:red, 0; green, 0; blue, 0 }  ][line width=0.75]      (0, 0) circle [x radius= 3.35, y radius= 3.35]   ;
\draw [shift={(332,249)}, rotate = 23.95] [color={rgb, 255:red, 0; green, 0; blue, 0 }  ][fill={rgb, 255:red, 0; green, 0; blue, 0 }  ][line width=0.75]      (0, 0) circle [x radius= 3.35, y radius= 3.35]   ;
\draw  [fill={rgb, 255:red, 255; green, 255; blue, 255 }  ,fill opacity=1 ] (382,257) -- (427.33,257) -- (427.33,274.8) -- (382,274.8) -- cycle ;
\draw [color={rgb, 255:red, 74; green, 144; blue, 226 }  ,draw opacity=1 ]   (331,162) -- (481.33,249.8) ;
\draw [color={rgb, 255:red, 74; green, 144; blue, 226 }  ,draw opacity=1 ]   (480.33,162.8) -- (332,249) ;
\draw [color={rgb, 255:red, 208; green, 2; blue, 27 }  ,draw opacity=1 ]   (331,70) -- (481.33,249.8) ;
\draw [color={rgb, 255:red, 208; green, 2; blue, 27 }  ,draw opacity=1 ]   (480.33,70.8) -- (332,249) ;
\draw  [fill={rgb, 255:red, 255; green, 255; blue, 255 }  ,fill opacity=1 ] (384,168) -- (429.33,168) -- (429.33,185.8) -- (384,185.8) -- cycle ;
\draw  [color={rgb, 255:red, 74; green, 144; blue, 226 }  ,draw opacity=1 ] (181,123) -- (201.33,128.4) -- (181,133.8) ;
\draw  [color={rgb, 255:red, 208; green, 2; blue, 27 }  ,draw opacity=1 ][fill={rgb, 255:red, 208; green, 2; blue, 27 }  ,fill opacity=1 ] (179,144) -- (205.33,151.4) -- (179,158.8) -- (192.17,151.4) -- cycle ;
\draw  [color={rgb, 255:red, 126; green, 211; blue, 33 }  ,draw opacity=1 ] (173,168) -- (192.33,173.9) -- (173,179.8) ;
\draw  [color={rgb, 255:red, 126; green, 211; blue, 33 }  ,draw opacity=1 ] (183,168) -- (202.33,173.9) -- (183,179.8) ;
\draw  [color={rgb, 255:red, 74; green, 144; blue, 226 }  ,draw opacity=1 ] (391,65) -- (411.33,70.4) -- (391,75.8) ;
\draw  [color={rgb, 255:red, 74; green, 144; blue, 226 }  ,draw opacity=1 ] (384.8,187.81) -- (400.1,202.25) -- (379.65,197.3) ;
\draw  [color={rgb, 255:red, 74; green, 144; blue, 226 }  ,draw opacity=1 ] (375.68,215.61) -- (396.45,212.26) -- (380.08,225.47) ;
\draw  [color={rgb, 255:red, 208; green, 2; blue, 27 }  ,draw opacity=1 ][fill={rgb, 255:red, 208; green, 2; blue, 27 }  ,fill opacity=1 ] (360,155) -- (386.33,162.4) -- (360,169.8) -- (373.17,162.4) -- cycle ;
\draw  [color={rgb, 255:red, 208; green, 2; blue, 27 }  ,draw opacity=1 ][fill={rgb, 255:red, 208; green, 2; blue, 27 }  ,fill opacity=1 ] (359.23,92.59) -- (370.73,117.4) -- (347.98,102.21) -- (362.17,107.4) -- cycle ;
\draw  [color={rgb, 255:red, 208; green, 2; blue, 27 }  ,draw opacity=1 ][fill={rgb, 255:red, 208; green, 2; blue, 27 }  ,fill opacity=1 ] (343.59,220.35) -- (367.56,207.17) -- (353.96,230.91) -- (358.17,216.4) -- cycle ;
\draw  [color={rgb, 255:red, 126; green, 211; blue, 33 }  ,draw opacity=1 ] (357.34,139.14) -- (376.81,133.69) -- (363.7,149.08) ;
\draw  [color={rgb, 255:red, 126; green, 211; blue, 33 }  ,draw opacity=1 ] (389,244) -- (408.33,249.9) -- (389,255.8) ;
\draw  [color={rgb, 255:red, 126; green, 211; blue, 33 }  ,draw opacity=1 ] (398,244) -- (417.33,249.9) -- (398,255.8) ;
\draw  [color={rgb, 255:red, 126; green, 211; blue, 33 }  ,draw opacity=1 ] (369.68,85.3) -- (381.29,101.85) -- (362.42,94.6) ;
\draw  [color={rgb, 255:red, 126; green, 211; blue, 33 }  ,draw opacity=1 ] (376.68,90.3) -- (388.29,106.85) -- (369.42,99.6) ;
\draw  [color={rgb, 255:red, 126; green, 211; blue, 33 }  ,draw opacity=1 ] (363.34,135.14) -- (382.81,129.69) -- (369.7,145.08) ;

\draw (172,188.4) node [anchor=north west][inner sep=0.75pt]    {$r-3$};
\draw (387,78.4) node [anchor=north west][inner sep=0.75pt]    {$r-3$};
\draw (389,168.4) node [anchor=north west][inner sep=0.75pt]    {$r-3$};
\draw (388,257.4) node [anchor=north west][inner sep=0.75pt]    {$r-3$};

\end{tikzpicture}

\caption{The graph $\Theta_r$ and its degree-$3$ cover obtained from~$K_{3,3}$.}
\label{fig:graph}
\end{figure}

Given an integer $s \geq 3$, let $\Theta_s$ be the graph consisting of two vertices and $s$ edges joining those vertices, and let $\iota_s$ be the involution of $\Theta_s$ that interchanges the vertices, and inverts each edge of, $\Theta_s$. Then $\Theta_r$ is the dual graph to the decomposition of~$\hat{M}$ along~$\partial \hat{M}_\mathrm{cut}$, and $\iota_r$ the involution of $\Theta_r$ induced by $\hat{\tau}$.
Now let $\iota_{3,3}$ be an automorphism of the complete bipartite graph $K_{3,3}$ that interchanges the two partite sets.
Then the fixed point set of $\iota_{3,3}$ in (the topological realization of) $K_{3,3}$ is nonseparating, and there is a degree-$3$ covering map $K_{3,3} \rightarrow \Theta_3$ with respect to which $\iota_3$ lifts as $\iota_{3,3}$. We can now obtain from~$K_{3,3}$ a degree-$3$ cover of $\Theta_r$ to which $\iota_r$ lifts as an involution whose fixed point set is nonseparating by adding $r-3$ edges joining each pair of vertices of~$K_{3,3}$ interchanged by $\iota_{3,3}$; see Figure \ref{fig:graph}. The corresponding cover of $\hat{M}$ is then as desired.
\end{proof}

As is apparent from the proof, Proposition \ref{prop:negativelycurved} holds under milder assumptions than negative curvature, but we chose to phrase the statement in that form given the longstanding open question of Gromov \cite{MR919829} as to whether indeed every Gromov-hyperbolic group is residually finite. In light of the ideas that went into the proof of the virtual Haken conjecture (see \cite{MR3363584} and the references therein), it seems more plausible that the fundamental group of any closed negatively curved manifold is residually finite under the assumption that the manifold contains a closed embedded totally geodesic hypersurface.

\subsection*{Acknowledgements} We thank Amina Abdurrahman, Misha Belolipetsky, Ian Biringer, Mingkun Liu, Bram Petri, and Richard Schwartz for helpful conversations and for their comments on an earlier draft of this note. We also thank the organizers of the event {\em Hyperbolic Manifolds, Their Submanifolds and Fundamental Groups} at IMPA, where these ideas were discussed. Finally, we are grateful to the referee for their comments on our manuscript. SD was supported by the Huawei Young Talents Program. FVP was funded by European Union (ERC, RaConTeich, 101116694).\footnote{Views and opinions expressed are however those of the author(s) only and do not necessarily reflect those of the European Union or the European Research Council Executive Agency. Neither the European Union nor the granting authority can be held responsible for them.}

\bibliography{reflectionsbib}{} 
\bibliographystyle{siam}

\end{document}